\documentclass[a4p,12pt]{article} 
\usepackage{graphicx}
\usepackage{amsmath}
\usepackage{amsthm, amssymb}
\usepackage{times}

\usepackage{bm}
\usepackage{layout}

\def\title{Hybridized Discontinuous Galerkin \par \medskip Method with Lifting Operator}
\def\abstract{
 \quad  In this paper, we propose a new hybridized discontinuous Galerkin method for the Poisson equation with
homogeneous Dirichlet boundary condition. Our method has the advantage that  
the stability is better than the previous hybridized method.
We derive $L^2$ and $H^1$  error estimates of optimal order.
Some numerical results are presented to verify our analysis.}
\def\keywords{discontinuous Galerkin method,  hybridized method, error analysis}
\newcounter{const}

\newtheorem{theorem}{Theorem}

\newtheorem{proposition}[theorem]{Proposition}

\newtheorem{remark}[theorem]{Remark}

\def\ang#1{\langle #1\rangle}

\makeatletter
\def\|{\@ifnextchar|{|\!\!\>|\!\!\>}{|\!\!\;|}}

\makeatother

\def\figurename{\textrm{Fig.}}
\def\tablename{\textrm{Table}}

\begin{document}

%\maketitle
\begin{center}
{\Large {\bf \title}}

\bigskip
\bigskip
{\large Issei Oikawa}

\bigskip

Graduate School of Mathematical Sciences, University of Tokyo \par
3-8-1 Komaba Meguro-ku Tokyo 153-8914, Japan

\bigskip
\bigskip

{\bf Abstract}
\end{center}
\vspace{-0.5em}
\abstract
\begin{flushleft}
\end{flushleft}
\vspace{2em}

\section{Introduction}
The discontinuous Galerkin finite-element methods (DGFEMs) is one
of the active research fields of numerical analysis in the last decade.
They allow us to use discontinuous approximate functions across the
element boundaries and have the robustness to variation of element
geometry. That is, we can utilize many kind of polynomials as approximate
functions on elements and many kind of polyhedral domains
as elements simultaneously. Consequently, DGFEM fits adaptive computations,
so that mathematical analysis as well as actual applications
has been developed for various problems. For more details, we refer
to \cite{Brenner,Arnold01,Arnold02}. However, the size and band-widths of the resulting matrices
can be much larger than those of the conventional FEM, which is a
disadvantage from the viewpoint of computational cost. To surmount
this obstacle, recently new class of DGFEM, which is called hybridized
DGFEMs, is proposed and analyzed by B. Cockburn and his colleagues;
for example, see \cite{Cockburn01}. Thus, we introduce new unknown
function $\hat U_h$ on inter-element edges and characterize it as the weak solution
of a target PDE. We then obtain the discrete system for $\hat U_h$ and
the size of the system becomes smaller. On the other hand, it should
be kept in mind that DGFEM has another origin. Some class of nonconforming
and hybrid FEM's, which are called hybrid displacement
method, use discontinuous functions as approximate field functions;
see for example \cite{Pian, Tong}. In \cite{KA72a} and \cite{KA72b}, F. Kikuchi and Y.
Ando developed a variant of the hybrid displacement one, and applied
it to plate problems. Their approach enables one to use conventional
element matrices and vectors. It, however, suffered from numerical
instability and was not fully successful. Recently, the author and his
colleagues proposed a new DGFEM that is based on the hybrid displacement
approach by stabilizing their old method and applied it to
linear elasticity problems in \cite{Kikuchi01}. A key point of our method is to introduce
penalty terms in order to ensure the stability. We, then, carried
out theoretical analysis by using the 2D Poisson equation as a model
problem, and gave some concrete finite element models with numerical
results and observations in \cite{Oikawa01}. However, an issue still remains. The stability
is guaranteed only when the penalty parameters are taken from
a certain interval, and we know only the existence of such an interval
and do not know concrete information about it.

The purpose of this paper is to propose a new hybridized DGFEM that
is stable for arbitrary penalty parameters. Our strategy is to introduce
the lifting operator and define the penalty term in terms of the lifting
operator. In order to state our idea as clearly as possible, we consider
the Poisson equation with
homogeneous Dirichlet condition: 
\begin{equation} \label{modelproblem}
  -\Delta u = f \text{ in } \Omega, \quad
          u = 0 \text{ on } \partial \Omega,
\end{equation} 
where $\Omega$ is a convex polygonal domain and $f \in L^2(\Omega)$.

This paper is composed of six sections. In Section 2, we introduce the
triangulation and finite element spaces, and then describe the lifting
operator. Section 3 is devoted to the formulation of our proposed hybridized
DGFEM, and mathematical analysis including error estimates
is given in Section 4. In Section 5, we report some results of numerical
computations and confirm our theoretical results. Finally, we conclude
this paper in Section 6.

\section{Preliminaries}

\subsection{Notation}

Let $\Omega \subset \mathbb R^n$, for an integer $n \ge 2$, be a convex polygonal domain.
 We introduce a triangulation $\mathcal T_h  = \{ K \}$  of $\Omega$ in the sense \cite{Oikawa01}, 
where  $h = \max_{K \in \mathcal T_h} h_K$ and  $h_K$ stands for the diameter of $K$.
 That is each $K \in \mathcal T_h$ is an $m$-polygonal domain, 
where $m$ is an integer and can differ with $K$. 
 We assume that $m$ is bounded from above independently of 
a family of triangulations $\{\mathcal T_h\}_h,$ and $\partial K$ does not intersect with itself.
 Let $\mathcal E_h = \{e \subset \partial K : K \in \mathcal T_h\}$ be the set of all edges of elements, and
let $\Gamma_h = \bigcup_{K \in \mathcal T_h} \partial K$.
 We define  the so-called broken Sobolev space for $k \ge 0$, 
\[
     H^k(\mathcal T_h) = \{ v \in L^2(\Omega) : v|_K \in H^k(K) \quad  \forall K \in \mathcal T_h \}.
\]
Let $L^2_0(\Gamma_h)$ $=$  $\{ \hat v$ $\in$ $L^2(\Gamma_h)$ : $\hat v|_{\partial\Omega}$ $= 0$ \}. 
We introduce the inner products
\begin{align*}
     (u,v)_K &= \int_K uvdx  \quad  \text{ for } K \in \mathcal T_h, \\
     \langle \hat u, \hat v \rangle_e &= \int_e \hat u \hat v ds  \quad \text{ for } e \in \mathcal E_h.
\end{align*}
 The usual $m$-th order Sobolev seminorm and norm on  $K$ are denoted by $|u|_{m,K}$ and $\|u\|_{m,K}$, respectively.
We use  finite element spaces:
\[
      U_h   \subset H^2(\mathcal T_h) , \quad \hat U_h  \subset L^2_0(\Gamma_h). 
\]
In addition, we set $V_h = U_h \times \hat U_h$ and $V(h) =  H^2(\mathcal T_h) \times L^2_0(\Gamma_h)$.

\subsection{Lifting operators}
 We state the definition of the lifting operator which plays a crucial role
in our formulation and analysis.
 To this end, we fix $K \in \mathcal T_h$ and $e \subset \partial K$ for the time being, and set
\[
      U_h(K) = \{ w_h|_K : w_h \in U_h,  \
      \hat U_h(e) = \{ \hat w_h|_e : \hat w_h \in \hat U_h\}.
\]
Then, for any $\hat v \in L^2(e)$, there exists a unique $\mathbf u_h \in U_h(K)^n$ such that
\begin{equation} 
      (\mathbf u_h, \mathbf w_h)_K = \langle \hat v, \mathbf w_h \cdot \mathbf n_K \rangle_e,
       \ \forall \mathbf w_h \in U_h(K)^n,
\end{equation}
where $\mathbf n_K$ is the unit outward normal vector to $\partial K$. 
 The lifting operator $\mathbf L_{e,K}$ : $L^2(e) \rightarrow  U_h(K)^n$
is defined as $\mathbf L_{e,K}(\hat v) = \mathbf u_h $. 
 Thus,
\begin{equation} \label{lifting}
      (\mathbf L_{e,K}(\hat v), \mathbf w_h)_K = \langle \hat v, \mathbf w_h \cdot \mathbf n_K \rangle_e,
      \ \forall \mathbf w_h \in U_h(K)^n.
\end{equation}
Furthermore, we define $\mathbf L_{\partial K} = \sum_{e\subset \partial K} \mathbf L_{e,K}$.

\section{New hybridized DG scheme}
 This section is devoted to the presentation of our proposed hybridized DGFEM.
 Before doing so, we convert the Poisson problem \eqref{modelproblem} into a suitable weak form \eqref{w-form}. 
 A key idea is to introduce unknown functions on inter-element edges.
 First, multiplying both the sides of \eqref{modelproblem} by a test function $v \in U_h$ and
integrating over each $K \in \mathcal T_h$, we have by the integration by parts
\begin{equation} \label{HDG01}
     \sum_{K\in\mathcal T_h} \left[ (\nabla u, \nabla v)_K - \ang{\mathbf n_K \cdot \nabla u,v}_{\partial K} \right]
      = (f,v)
\end{equation}
 From the continuity of the flux, we have
\begin{equation} \label{conti-flux}
      \sum_{K \in \mathcal T_h} \ang{\mathbf n_K \cdot \nabla u, \hat v} = 0 \quad \forall \hat v \in L^2_0(\Gamma_h).
\end{equation}
This, together with \eqref{HDG01}, implies 
\begin{equation} \label{HDG02}
       \sum_{K\in\mathcal T_h} \left[ (\nabla u, \nabla v)_K - 
                                        \ang{\mathbf n_K \cdot \nabla u,v-\hat v}_{\partial K} \right]
              = (f,v)
\end{equation}
 Here we set, for $\mathbf u = (u,\hat u)$ and $\mathbf v = (v,\hat v) \in V(h)$,  
\begin{align*}
     a_h(\mathbf u, \mathbf v) &= \sum_{K\in\mathcal T_h} (\nabla u, \nabla v)_K, \\
     b_h(\mathbf u, \mathbf v) &= -\sum_{K\in\mathcal T_h} \ang{\mathbf n_K \cdot \nabla u, v-\hat v}_{\partial K}.
\end{align*}
Then, \eqref{HDG02} is rewritten as 
\begin{equation} \label{w-form}
     a_h(\mathbf u, \mathbf v) + b_h(\mathbf u, \mathbf v) = (f,v).  
\end{equation}
 Now we can state our hybridized DGFEM:
find $\mathbf u_h \in V_h$ such that
\begin{align} \label{HDGformulation}
       B^L_h(\mathbf u_h, \mathbf v_h)  
      & := a_h(\mathbf u_h, \mathbf v_h)
         + b_h(\mathbf u_h, \mathbf v_h)
         + b_h(\mathbf v_h, \mathbf u_h)
         + j_h(\mathbf u_h, \mathbf v_h)  \nonumber \\ 
      & \quad = (f,v_h) \quad \forall \mathbf v_h = (v_h, \hat v_h) \in V_h.
\end{align}

 Here, the third term $b_h(\mathbf v_h,  \mathbf u_h)$ of $B_h^L$ is 
added in order to symmetrize the scheme and the penalty term 
$j_h(\mathbf u_h, \mathbf v_h)$ is defined by
\begin{align*}
        j_h(\mathbf u, \mathbf v) =
     &  \sum_{K\in\mathcal T_h} \left( \mathbf L_{\partial K}(u - \hat u), \mathbf L_{\partial K}(v-\hat v) \right)_K \\
     & + \sum_{K \in\mathcal T_h}\sum_{e \subset \partial K} \int_e \eta_e h_e^{-1} (u - \hat u) (v - \hat v) ds,
\end{align*}
 with the penalty parameters  $\eta_e > 0$, where $h_e$ is the diameter of $e$.

\section{Error estimates}
 In this section, we give a mathematical analysis of our hybridized DGFEM.
 To this end, we introduce
\begin{align*}
      \|| \mathbf v \||^2  
      &= \sum_{K\in\mathcal T_h} \Bigg( \| \nabla v - \mathbf L_{\partial K}(v - \hat v)\|_{0,K}^2     
         + \sum_{e \subset \partial K} \frac{\eta_e}{h_e} \|v - \hat v\|^2_{0, e} \Bigg), \\
      \|| \mathbf v \||^2_h
       & = \sum_{K\in\mathcal T_h} \left( |  v |_{1,K}^2 + 
                                 \sum_{e \subset \partial K}  \frac{\eta_e}{h_e} \|v - \hat v\|^2_{0, e} \right),
\end{align*}
 where $\eta_e$ is a positive parameter for each $e \in \mathcal E_h$.

\begin{theorem} \label{thm:CBC}
 The bilinear form $B_h^L$ satisfies the following three properties.
\begin{description}
     \item[(Consistency)]
         Let  $ u \in H^2(\Omega) \cap H^1_0(\Omega)$
         be  the exact solution. For $\mathbf u = (u, u|_{\Gamma_h})$, we have
         \[
                B^L_h(\mathbf u, \mathbf v) = (f,v) \quad \forall \mathbf v \in V(h).
         \]
     \item[(Boundedness)] 
         \[
               |B^L_h(\mathbf v, \mathbf w)| \le \|| \mathbf v\|| \|| \mathbf w\||
                \quad \forall \mathbf v,\mathbf w \in V(h).
         \]
     \item[(Coercivity)] 
         \[
                B^L_h(\mathbf v_h, \mathbf v_h) \ge  \|| \mathbf v_h \||^2
                 \quad \forall \mathbf v_h \in V_h.
         \]
\end{description}
 Furthermore, the scheme \eqref{HDGformulation} admits a unique solution $\mathbf u_h \in V_h$
for any $f \in L^2(\Omega)$ and $\{\eta_e\}_e$. 
\end{theorem}
\begin{proof}
 The consistency is trivial since $u -  u|_{\Gamma_h} = 0$ on $\Gamma_h$. 
 The coercivity is a direct consequence of the expression
\[
      b_h(\mathbf v, \mathbf w) = -\sum_K (\nabla v, \mathbf L_{\partial K}(w - \hat w) )_K. 
\]
 Combining this with the Schwarz inequality, we immediately deduce the boundedness.
 Finally, the coercivity implies the uniqueness of \eqref{HDGformulation} and, hence, the 
system of linear equations \eqref{HDGformulation} admits a unique solution.
\end{proof}
 As results of those three properties, we obtain the following a priori error estimates in terms of $\||\cdot\||$.
\begin{theorem} \label{thm:err}
 Let $\mathbf u = (u, u|_{\Gamma_h}) \in V(h)$ with  the exact solution
$u \in H^2(\Omega) \cap H^1_0(\Omega)$ of the Poisson problem \eqref{modelproblem}.
 Suppose that $\{\mathcal T_h\}_h$ satisfies
\begin{equation} \label{inv}
     \tau \le \frac{h_e}{h_K} \quad \forall K \in \mathcal T_h, \forall e \subset \partial K
\end{equation}
 with some positive constant $\tau$.
 Let $\mathbf u_h = (u_h, \hat u_h) \in V_h$ be the solution of 
our HDG scheme \eqref{HDGformulation} for an arbitrary $\{\eta_e\}_e$, $\eta_e>0$.
 Then, we have the error estimates
\begin{align} \label{err-infH1}
     \|| \mathbf u - \mathbf u_h \|| \le 2 \inf_{v_h \in V_h} \|| \mathbf u - \mathbf v_h \||.
\end{align}
\end{theorem}

\begin{proof}
 Let $\mathbf v_h \in V_h$ be arbitrary.  By Theorem \ref{thm:CBC}, we have
\begin{align*}
      \|| \mathbf u_h - \mathbf v_h \||^2 
     &  \le B_h^L(\mathbf u_h - \mathbf v_h, \mathbf u_h - \mathbf v_h)  & \text{(Coercivity)} \\
     &  =   B_h^L(\mathbf u - \mathbf v_h, \mathbf u_h - \mathbf v_h)   & \text{(Consistency)} \\
     &  \le  \||\mathbf u - \mathbf v_h\|| \ \|| \mathbf u_h - \mathbf v_h \||, & \text{(Boundedness)}
\end{align*} 
 which implies that 
\begin{align}
     \|| \mathbf u_h - \mathbf v_h \|| \le  \|| \mathbf u - \mathbf v_h\|| \quad \forall \mathbf v_h \in V_h.
\end{align}
 Using the triangle inequality, we have
\begin{align*}
     \|| \mathbf u - \mathbf u_h \||
     \le \||\mathbf u - \mathbf v_h \|| + \|| \mathbf u_h - \mathbf v_h\|| 
     \le 2 \|| \mathbf u - \mathbf v_h\||.
\end{align*}
 From the above, it follows that 
\begin{equation}
    \|| \mathbf u - \mathbf u_h \|| \le  2 \inf_{\mathbf v_h \in V_h }\|| \mathbf u - \mathbf v_h \||,
\end{equation}
 which implies that the error of the approximate solution is optimal in the norm $\||\cdot\||$. 
\end{proof}
As is stated in \cite{Oikawa01}, we assume that the following approximate properties:
for $v \in H^{k+1}(K)$ there exist positive  constants $C^{\rm e}_{k,s}$ and $C^{\rm f}_{k,s}$ such that
\begin{align}
    \inf_{v_h \in U_h} |v - v_h|_{s,K} 
        \le C^{\rm e}_{k,s} h_K^{k+1-s} |v|_{k+1,K},  \label{approx-elt}\\
    \inf_{\hat  v_h \in \hat U_h}   | v - \hat v_h |_{s,e}
        \le C^{\rm f}_{k,s} h_K^{k+\frac 1 2 -s} |v|_{k+1,K}. \label{approx-face}
\end{align}
 Then we have the error estimates in Theorem \ref{thm:err} are actually of optimal order.

\begin{theorem} \label{thm:err2}
 Under the assumptions in Theorem \ref{thm:err} and the approximate properties 
\eqref{approx-elt} and \eqref{approx-face}, we have, if $u \in H^{k+1}(\Omega) \cap H^1_0(\Omega)$, 
\begin{align}
    \|| \mathbf u- \mathbf u_h \|| \le Ch^k | u |_{k+1,\Omega},  \label{err-H1}\\
    \|  u-  u_h \|_{0,\Omega}  \le Ch^{k+1} | u |_{k+1,\Omega}.             \label{err-L2}       
\end{align} 
\end{theorem}
 In order to prove Theorem \ref{thm:err2}, we need the following auxiliary result.
\refstepcounter{const} \label{lift1}
\begin{proposition} \label{prop02}
 Let $K \in \mathcal T_h$ and $e \subset \partial K$.
 Then we have
\begin{equation}  \label{prop02-1}
    \| \mathbf L_{e,K}(\hat v) \|_{0,K} \le C_{\ref{lift1}}  h_e^{-1/2} \| \hat v \|_{0,e}
    \ \forall \hat v \in L^2(e).    
\end{equation}
\end{proposition}
\begin{proof}

In \eqref{lifting}, taking $\mathbf w_h =  \mathbf L_{e,K}(\hat v)$ yields 
\begin{align}
     \|  \mathbf L_{e,K}(\hat v) \|_{0,K}^2  
    &= ( \mathbf L_{e,K}(\hat v) ,  \mathbf L_{e,K}(\hat v) )_K    \nonumber \\
    &= \langle \hat v,  \mathbf L_{e,K}(\hat v) \rangle_e           \nonumber  \\ 
    &\le \| \hat v \|_{0,e} \| \mathbf L_{e,K}(\hat v) \|_{0,e}.        \label{prop02-prf1}
\end{align}
 By the trace theorem, there exists $C_1$ such that
\begin{equation} \label{prop02-prf2}
     \| \mathbf L_{e,K}(\hat v) \|_{0,e} \le C_1h_e^{-1/2} \| \mathbf L_{e,K}(\hat v)\|_{0,K}.
\end{equation}
 Here $C_1$ depends on $U_h(K)$ and $\hat U_h(e)$.
 Combining \eqref{prop02-prf1} with \eqref{prop02-prf2}, we obtain \eqref{prop02-1}.
\end{proof}

\begin{proof}[Proof of Theorem \ref{thm:err2}]
\refstepcounter{const} \label{eqv-norm}
 As a consequence of  Proposition \ref{prop02},
it can be proved that there exists a constant $C_{\ref{eqv-norm}}$ such that
\begin{equation} \label{norm}
   \|| \mathbf v \|| \le C_{\ref{eqv-norm}} \|| \mathbf v \||_h \quad \forall \mathbf v \in V(h).
\end{equation}
 From  \eqref{approx-elt} and \eqref{approx-face}, 
 we have   %\refstepcounter{const} \label{ord-h-norm}
\begin{equation} 
    \inf_{\mathbf v_h \in V_h}  \|| \mathbf u - \mathbf v_h \||_h 
     \le  C h^k | u |_{k+1,\Omega}.
\end{equation}
 Combining this with \eqref{norm}, we obtain \eqref{err-H1}.
 Next, we prove \eqref{err-L2}.
 Here we define $\psi \in H^2(\Omega) \cap H^1_0(\Omega)$
 as the solution of the adjoint problem
\begin{equation}
    -\Delta \psi = u - u_h \text{ in } \Omega, \quad \psi = 0 \text{ on } \partial\Omega. 
\end{equation}
 Let   $\boldsymbol \psi = (\psi,\psi|_{\Gamma_h})$.  Then,  since $B_h^L$ is symmetric,  we have
\begin{equation}
     B_h^L(\mathbf v, \boldsymbol\psi) = (u-u_h, v) \quad \forall \mathbf v = (v,\hat v) \in V(h).
\end{equation}
 In particular, taking $\mathbf v = \mathbf u - \mathbf u_h$, we have for any $\boldsymbol \psi_h \in V_h$,
\begin{align*}
     \| u - u_h \|_{0,\Omega}^2
    & \le B_h^L(\mathbf u-\mathbf u_h, \boldsymbol\psi ) \\
    & = B_h^L(\mathbf u-\mathbf u_h, \boldsymbol\psi - \boldsymbol \psi_h )  \\
    & \le \||\mathbf u-\mathbf u_h \|| \||\boldsymbol\psi - \boldsymbol \psi_h \|| \\
    & \le C_{\ref{eqv-norm}} \||\mathbf u-\mathbf u_h \||  \||\boldsymbol\psi - \boldsymbol \psi_h \||_h.
\end{align*}
 From \eqref{approx-elt} and \eqref{approx-face}, it follows that 
\begin{equation}
     \||\boldsymbol\psi - \boldsymbol \psi_h \||_h \le C h |\psi|_{2,\Omega}.
\end{equation}
 By the regularity of the adjoint problem, we have
\begin{equation}
      |\psi|_{2,\Omega} \le C \| u - u_h \|_{0,\Omega}.
\end{equation}
 Thus we obtain \eqref{err-L2}.
\end{proof}
\begin{remark} {\rm
 In contrast to our previous results of \cite{Oikawa01},
error estimates in Theorem \ref{thm:err} are valid for any positive parameters $\eta_e$.
 This is one of the advantages of our hybridized DGFEM.
}\end{remark}

\section{Numerical results}

 We now present the numerical results of our method for the following Poisson equation:
\begin{equation}
    \begin{cases}
         & -\Delta u =2\pi^2\sin(\pi x)\sin(\pi y)  \text{ in } \Omega, \\
         &  \quad        u = 0 \text{ on } \partial\Omega,
    \end{cases}
\end{equation}
where $\Omega$ is a unit square.
 We use uniform rectangular meshes and $P_k$--$P_k$ elements ($k=1,2,3$). 
 We computed the approximate solutions for various mesh size $h = 1/N$, see \tablename \  \ref{tbl-error}.
 We take the unity as the penalty parameters for each $e \in \mathcal E_h$.
 We see from \tablename \ \ref{tbl-error} that the $H^1$ and $L^2$ convergence rate of the approximate solutions
are $h^k$ and $h^{k+1}$, respectively.
 \figurename \ref{fig1} and \figurename \ref{fig2} show the approximate solution $u_h$ and $\hat u_h$  in the case  $k=1$ and $N=8$,
respectively. 
\renewcommand\arraystretch{1.2}
\begin{table}[htbp]
\caption{$L^2$ and $H^1$ errors.}
\label{tbl-error}
\smallskip
\center
\begin{tabular}{cccccc} \hline
      &   & \multicolumn{2}{c}{$L^2$}          &  \multicolumn{2}{c}{$H^1$}    \\  \cline{3-6}
 $k$  & $N$   & error    & rate   &  error  & rate  \\ \hline \hline
1 & 4 & 3.23E-02 & 1.96  & 7.15E-01 & 1.01  \\
  & 8 & 8.29E-03 & 1.96  & 3.55E-01 & 1.00  \\
  & 16 & 2.14E-03 & 1.99  & 1.78E-01 & 1.00  \\
  & 32 & 5.39E-04 &  & 8.90E-02 &  \\ \hline
2 & 4 & 4.56E-03 & 3.18  & 1.46E-01 & 2.07  \\
  & 8 & 5.04E-04 & 3.05  & 3.47E-02 & 2.02  \\
  & 16 & 6.08E-05 & 3.01  & 8.58E-03 & 2.00  \\
  & 32 & 7.53E-06 &  & 2.14E-03 &  \\ \hline
3 & 4 & 4.48E-04 & 4.21  & 2.00E-02 & 3.12  \\
  & 8 & 2.43E-05 & 4.07  & 2.30E-03 & 3.03  \\
  & 16 & 1.45E-06 & 4.02  & 2.81E-04 & 3.01  \\
  & 32 & 8.94E-08 &  & 3.49E-05 &  \\ \hline
\end{tabular}
\end{table}

\begin{figure}[!b]
\begin{center} 
\includegraphics[width=.8\textwidth]{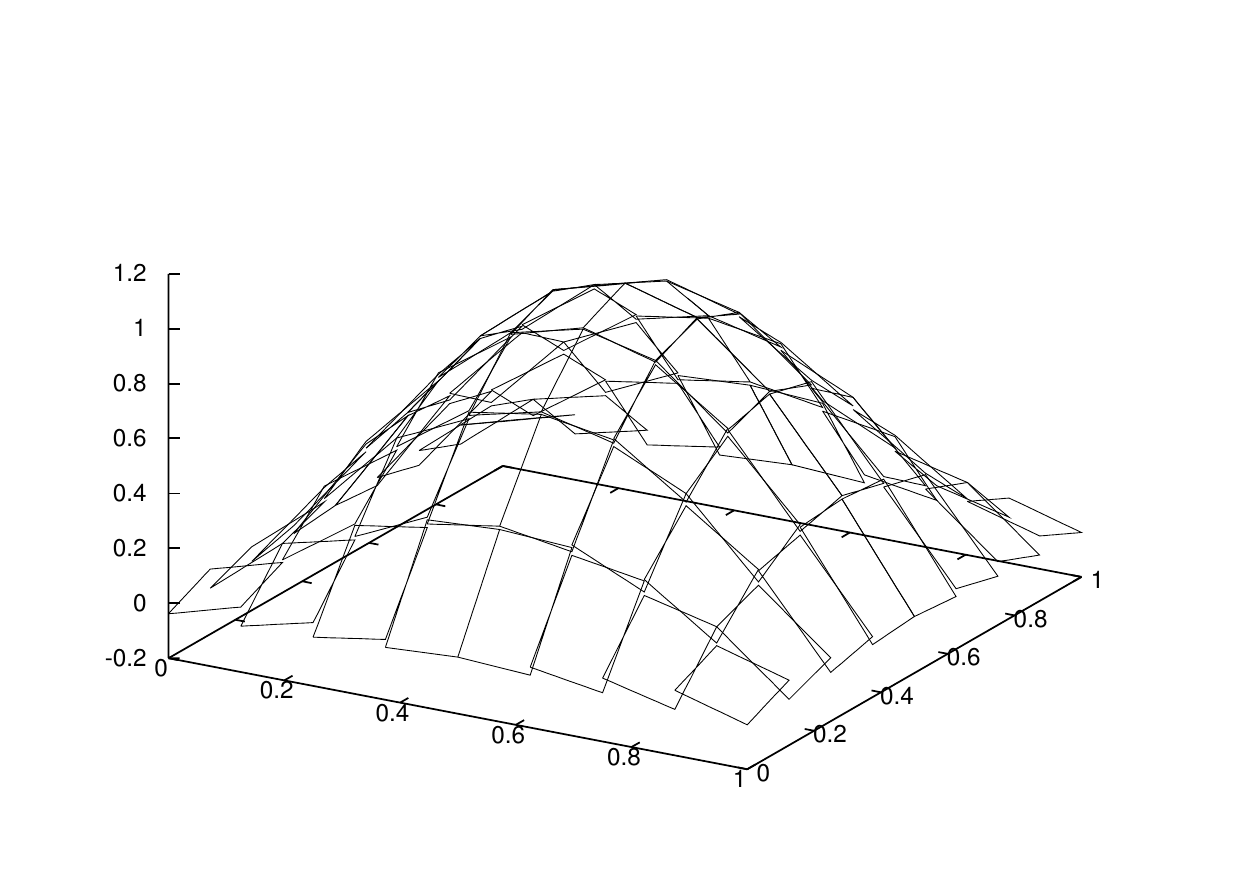}
\end{center}
\caption{The approximate solution $u_h$  in the case $k=1$ and $N=8$.}
\label{fig1}
\end{figure}
%

%\vspace{-3em}
\begin{figure}[!h]
\begin{center} 
\includegraphics[width=.8\textwidth]{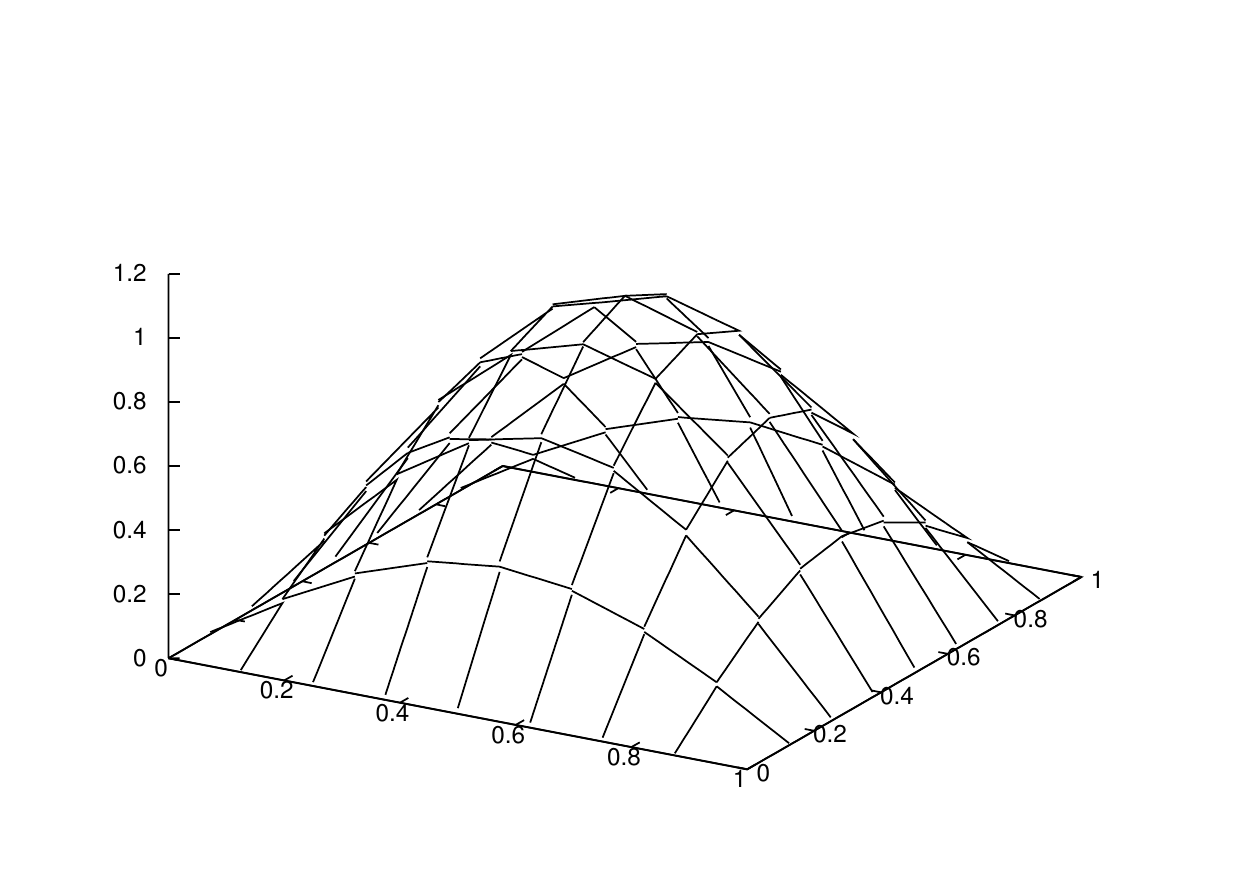}
\end{center}
\caption{The approximate solution $\hat u_h$ in the case $k=1$ and $N=8$.}
\label{fig2}
\end{figure}

\section{Conclusions}
 We have presented a new hybridized DGFEM by using the lifting operator
and examined the stability for arbitrary penalty parameters. 
 Convergence results of optimal order have been proved and confirmed by numerical experiments.
 As a model problem, we have considered only the Dirichlet boundary value problem for the Poisson equation.
 We are interested in application to other problems, for example, 
Neumann boundary value problem, convection-diffusion equations, Stokes system,
and time-dependent problems. They are left here as future study.

\section{Acknowledgement}
I thank Professor Fumio Kikuchi who brought
my attention to the present subject and encouraged me through valuable
discussions. This work is supported by Grants-in-Aid for Scientific
Research, JSPS and by Global COE Program (The Research and
Training Center for New Development in Mathematics, The University
of Tokyo), MEXT, Japan.

%%%%% FIG. 1. %%%%%%}

%
%
%%%%% FIG. 2. %%%%%%

%\references

\end{document}